\documentclass[10pt, reqno]{amsart}

\usepackage[utf8]{inputenc}
\usepackage{amsmath}
\usepackage{amsthm}
\usepackage{accents}
\usepackage{enumerate}
\usepackage{amsfonts}
\usepackage{verbatim}
\usepackage[foot]{amsaddr}
\usepackage{graphicx}
\usepackage{float}
\usepackage[margin=3.5cm]{geometry}

\theoremstyle{plain}
\newtheorem{theorem}{Theorem}

\newtheorem{corollary}{Corollary}
\newtheorem{lemma}{Lemma}

\theoremstyle{remark}
\newtheorem*{remark}{Remark}

\theoremstyle{definition}

\begin{document}
	
	       \author[Sebastian Rosengren]{Sebastian Rosengren$^1$}      
	       \address{$^{1}$Department of Mathematics, Stockholm University, 106 91 Stockholm, Sweden.}
	       \email{$^1$rosengren@math.su.se}

\title[Multi-type Preferential Attachment]{A Multi-type Preferential Attachment Tree} 
\maketitle

\begin{abstract}
	A multi-type preferential attachment tree is introduced, and studied using general multi-type branching processes. For the $p$-type case we derive a framework for studying the tree where a type $i$ vertex generates new type $j$ vertices with rate $w_{ij}(n_1,n_2,\ldots, n_p)$ where $n_k$ is the number of type $k$ vertices previously generated by the type $i$ vertex, and $w_{ij}$ is a non-negative function from $\mathbb{N}^p$ to $\mathbb{R}$. The framework is then used to derive results for trees with more specific attachment rates.
	
	In the case with linear preferential attachment---where type $i$ vertices generate new type $j$ vertices with rate $w_{ij}(n_1,n_2,\ldots, n_p)=\gamma_{ij}(n_1+n_2+\dots +n_p)+\beta_{ij}$, where $\gamma_{ij}$ and $\beta_{ij}$ are positive constants---we show that under mild regularity conditions on the parameters $\{\gamma_{ij}\}, \{\beta_{ij}\}$ the asymptotic degree distribution of a vertex is a power law distribution. The asymptotic composition of the vertex population is also studied.\\
\end{abstract}


\section{Introduction}
The preferential attachment tree is a well-studied model of random network growth where, traditionally, new vertices arrive according to some process (often at integer times $\{1,2,\ldots\}$) and upon arrival attach randomly to an existing vertex with probability proportional to that vertex's degree. This may also be interpreted as the existing vertex giving birth to a new vertex. Other versions exists where arriving vertices attach to \textit{more} than one vertex. This gives rise to preferential attachment networks while the special cases considered here results in preferential attachment trees.

Preferential attachment is particularly interesting for modeling empirical networks and trees since it gives rise to power law degree distributions---something that is often observed in network data, see e.g. \cite{barabasi} and \cite{hofstad_2016} where many applications of preferential attachment models can be found. In the context of network modeling the preferential attachment mechanism was first introduced in \cite{barabasi} and the degree sequence was rigorously analyzed in \cite{Bollobas}.

In this paper we extend the preferential attachment tree by allowing vertices to be of one of $p$ different types, where vertices of different types give birth to new vertices at different rates. A very general application of this dynamic is in recruitment networks, where the edges represent a recruitment. Different types could be male-female, or political affiliations. Males may then recruit males at a different rate than females. Another example could be a male-female network where edges represent friendship. In order for this to be a tree however, a friendship link could represent e.g. the first friend a person makes upon entering the graph. These two examples fit into our model framework, as described in more detail below.
More generally, the model can be used to describe multi-type trees exhibiting some level of homophily or heterophily.

Traditionally, preferential attachment is studied in discrete time, but we will develop the framework for studying the model in continuous time. The reason for this is that the model can then be analyzed as a general branching process in continuous time and existing results on such processes can then be applied. The downside of this approach is that applications are restricted to trees, something that is not always realistic. However, in a recent paper \cite{remco2} the authors shows a way of collapsing a branching process resulting in a (specific) preferential attachment network, while still being able to apply powerful branching process results. This is done for the single type case with affine rate functions, and extensions to the multi-type case seems possible. We will not use this approach, but instead stick to trees allowing general rate functions.

In studying the model we shall mainly be interested in the asymptotic degree distribution of a vertex. Usually, in one-type preferential attachment this is analyzed by first deriving a recursion for the \textit{expected} fraction of vertices having degree $k$ and showing that this converges. Stronger convergence results follow with an additional martingale argument, see \cite{Bollobas} for a rigorous treatment. The drawback of this method is that it depends heavily on the linear structure of the attachment dynamic (new vertices attach proportional to degree, and not to an arbitrary function of the degree), and it tends to be difficult to apply the method to more general preferential attachment models, e.g. multi-type preferential attachment with non-linear attachment functions.

An alternative way of studying preferential attachment models is to embed the process in continuous time and interpret it as a general branching process. This was first done by Rudas et al. \cite{RSA} and later extended by Deijfen \cite{Deijfen} who also allowed for vertex death. In this article we build on the methods used in \cite{RSA} and extend them to the multi-type case. Deijfen and Fitzner \cite{Deijfen2} have studied two-type preferential attachment heuristically, using different methods. There an arriving vertex is of type $i$ with probability $p_i$, then chooses which type $j$ to attach to with probability $\theta_{i,j}$ and then attaches to a vertex of that type with probability proportional to degree. In the context of multi-type graphs we also mention the seminal paper \cite{Jansson}, which treats a very general model that includes an instance resembling preferential attachment.
\subsection{The Model}
\label{modeldef}
We begin by defining the multi-type degree-based attachment tree, of which the multi-type preferential attachment tree is a special case. 
In the multi-type degree-based attachment tree a vertex may be of one of $p$ different types. The vertex population evolves in continuous time where new vertices are born, but may not die.
A type $i$ vertex currently having $n_k$ type $k$ children, $k=1,2,\ldots,p$ gives birth to a new type $j$ child at rate $w_{ij}(\vec{n})$, where $\vec{n} = (n_1,\ldots, n_p ) \in \mathbb{N}^p $---i.e. given the number of children $\vec{n}$ of respective types of a type $i$ vertex the time until the next birth is exponentially distributed with rate $w_{i1}(\vec{n})+\dots+w_{ip}(\vec{n})$ and the birth is of type $j$ with probability $\frac{w_{ij}(\vec{n})}{w_{i1}(\vec{n})+\dots+w_{ip}(\vec{n})}$. This process is then turned into a graph by letting the relation mother-child be represented by a directed edge from child to mother. The number of children of a vertex is then the vertex's in-degree, or total degree minus 1. The graph starts at time $t=0$ with \textit{one} vertex of any type present.

For a formal definition we let $\{\xi_{ij}(\cdot)\}$ $i,j=1,\ldots,p$ be a point process on $\mathbb{R_+}$, and $\xi_{ij}(t) = \xi_{ij}([0,t])$ the counting process associated with it. The process $\{\xi_{ij}(\cdot)\}$ is to be interpreted as the number of type $j$ vertices that a type $i$ vertex gives birth to in some time interval. 
Now, define the multi-type degree-based attachment tree as the stochastic process evolving according to the following dynamics (where births are represented as directed edges),

\begin{enumerate}[(i)]
	\item The weight functions satisfy $w_{ij}: \mathbb{N}^p\to (0, \infty),$  $i,j=1,2,\ldots,p$.
	\item $\{ \xi_{ij}(t),\ t\geq 0 \}$ is a counting process with $\xi_{ij}(0) = 0$, $i,j=1,2\ldots,p$.
	\item For $i=1,2,\ldots,p$ the process $\vec{\xi_i}(t) = (\xi_{i1}(t),\xi_{i2}(t),\ldots,\xi_{ip}(t))$ is a continuous-time Markov chain on $\mathbb{N}^p$ with transition rates
	$w_i(\vec{n}) = w_{i1}(\vec{n})+\dots+w_{ip}(\vec{n})$ and transition probabilities
	\\
	$p(\vec{n}\to (n_1,\ldots, n_j+1, \ldots,n_p )) = \frac{w_{ij}(\vec{n})}{w_{i1}(\vec{n})+\dots+w_{ip}(\vec{n})}$.
\end{enumerate}
This means that the vertex population evolves according to a multi-type general branching process, see \cite[Ch. 6]{NermanPHD} or \cite{multiypeextension}, and the graph can therefore be analyzed within that framework. When the weight functions are increasing in the type specific degrees we shall instead refer to the model as the multi-type preferential attachment tree, and reserve the pronoun degree-based for the case when the weight functions are more arbitrary. The preferential attachment tree is hence a special case of the degree-based attachment tree.

By looking at the graph at the times when new vertices arrive we can connect it to a discrete time degree-based attachment tree. Let $\mathcal{G}(t)$ denote the $\sigma$-field generated by the graph up until time $t$; and $\sigma_n$ the birth time of the $n$:th vertex. Also let $w^v(t)$ be the weight of the vertex $v$ at time $t$, i.e. $w^v(t) = \sum_{j=1}^{p} w_{ij}(\vec{\xi}_i(t))$ if $v$ is of type $i$. Similarly, let  $w_k^v(t) = w_{ik}(\vec{\xi}_i(t))$ if $v$ is of type $i$. Then the probability that the $(n+1)$:th new vertex arriving is of type $i$ given the current state of the graph is given by
\begin{align*}
p(i | \mathcal{G}(\sigma_n)) = \frac{\sum_{v=0}^{n} w_i^v(\sigma_n)}{ \sum_{v=0}^{n} w^v(\sigma_n) },
\end{align*}
where we have numbered vertices in the order they arrived. Given the current state of the graph the $(n+1)$:th vertex $v$ attaches to an old vertex $u$ with probability
\begin{align*}
p(v\to u | \mathcal{G}(\sigma_n) ) = \frac{ w^u(\sigma_n) }{\sum_{v=0}^{n} w^v(\sigma_n)}.
\end{align*}
For the one-type case $p=1$ with weight function $w(k)=k+1$ the graph tree that arises by inspecting the continuous time tree at the birth times coincides with the standard discrete time preferential attachment tree. This connection was noted already in \cite{RSA}.
\begin{remark}
	Using the connection between the continuous time model and the discrete time model we see that limit results for the continuous time case are valid for the discrete time case as long as $\sigma_n \to \infty$ when $n \to \infty$, i.e. the continuous time model does not explode in finite time.
\end{remark}
\subsubsection*{Two Examples}
In the introduction we suggested an application of the framework to model a network consisting of males and females where the edges represents friendship, e.g. the first friend a person makes when arriving to the network. Starting with one person, of any type, how does this example fit into the framework? Clearly, the number of friends a person has is an indicator of ones social skill and should determine the \textit{total} rate at which a person makes new friends. However, we assume that males make new male friends at a different rate than they make new female friends. For example, we can specify the model using the rate functions $w_{ij}(n_1,n_2) = \gamma_{ij}(n_1+n_2)$, $i,j=1,2$ ,where $1$ represents male and $2$ female, and $\gamma_{ij}$ are positive real numbers.

The final example was that of a recruitment graph. Take for instance the democratic and republican party in the U.S. A card-carrying democrat recruits democrats, to her party, at a different rates than republicans, and vice-versa. We can model the recruitment skill of a person in any way we like through $w_{ij}(n_1,n_2)$.

\subsection{Notation}
Finally, we set some notation. Again, $\vec{n} = (n_1,\ldots, n_p) \in \mathbb{N}^p$ and $\vec{\xi_i}(t) = (\xi_{i1}(t), \xi_{i2}(t), \ldots, \xi_{ip}(t)) \in \mathbb{N}^p $; the letters $i$ and $j$ will always refer to vertex type; the indexed letters $n_j$ will always refer to the number of type $j$ children of a vertex or, equivalently, its in-degree from type $j$ vertices. We will use the star notation $^*$ to denote the Laplace transform of a function or a matrix. Throughout we let $Z(t)$ denote the number of vertices in the graph at time $t$; $Z_i(t)$ the number of type $i$ vertices in the graph at time $t$; $Z_i^{\vec{n}}(t)$ the number of type $i$ vertices in the graph at time $t$ with $n_j$ type $j$ children, $j=1,\ldots,p$; $Z_i^{k}(t)$ the number of type $i$ vertices in the graph at time $t$ with $k$ children in total. Also, let $f\sim g$ denote that $\lim\limits_{t\to\infty} \frac{f(t)}{g(t)} = 1$.

\subsection{Results for Degree-based Preferential Attachment}
\label{notation}
The main results concern asymptotic composition of the population and degree distribution: we shall formulate conditions on $\{w_{ij}(\vec{n}) \}$ so that the ratios $p_i(t) = \frac{Z_i(t)}{Z(t)}$, $p_i(\vec{n}, t) = \frac{Z_i^{\vec{n}}(t)}{Z(t)}$, and  $p_i(k, t) = \frac{Z_i^{k}(t)}{Z(t)}$ converges almost surely as $t \to \infty$ and identify the limits.

For each Borel measurable set $A$ in $\mathbb{R}_+$ let $\mu_{ij}(A) = \mathbb{E}(\xi_{ij}(A))$---i.e. the expected number of type $j$ vertices born by a type $j$ vertex in the time set $A$. It follows that $\mu_{ij}(A)$ is a measure (see e.g. \cite[Lemma 1.1.1]{Point}) and we define for each $\theta >0$ the new measure $\mu_{ij}^{\theta}(A) = \mu_{ij}(A, \theta) $ on $\mathcal{B}(\mathbb{R}_+)$ through
\begin{align*}
\mu_{ij}(A, \theta) = \int_{A} e^{-\theta s} \mu_{ij}(ds), \qquad \theta >0.
\end{align*}

Define $\mu_{ij}^*(\theta) = \mu_{ij}([0,\infty], \theta)$ and the matrix $\mu^*(\theta) = [\mu_{ij}^*(\theta) ]_{ i,j=1,\ldots,p}$, and let $\rho(\mu^*(\theta))$ denote the largest eigenvalue of $\mu^*(\theta)$ (also known as the Perron-Frobenius root).
Throughout we shall assume the existence of a \textit{Malthusian} parameter $\alpha \in (0,\infty)$ such that $\rho(\mu^*(\alpha))=1$.
In fact we shall assume that
\begin{equation}
\label{A1}
\exists \theta_0 \in (0,\infty):\qquad \rho( \mu^*(\theta_0) ) \in( 1, \infty). \tag{A1}
\end{equation}
It follows from \eqref{A1} that the Malthusian parameter exists since $\rho(\theta)$ is continuous and decreasing to zero, see \cite[Lemma 9.1]{Mode}. By (A1) and \cite[Thm 6.1]{Mode} the process does not explode in finite time.

The Malthusian parameter $\alpha$ is, even for simple models, given by a very complicated expression. However, our assumptions ensure that it always exists and can be calculated numerically. For $\mu^*(\alpha)$ we denote the corresponding \textit{left} eigenvector by $u = (u_1,\ldots,u_p)$ and the \textit{right} eigenvector by $v=(v_1,\ldots,v_p)^t$. By the Perron-Frobenius theorem, both these exists, are positive, and can be normed so that
\begin{align*}
u_1v_1+u_2v_2+\dots+u_pv_p = v_1+v_2+\dots+v_p = 1.
\end{align*}
We will throughout assume that the eigenvectors are normed in this way.

We can now state our main results. The first one concerns the asymptotic composition of the vertex population.
\begin{theorem}
	\label{thm1}
	For the multi-type degree-based attachment model starting with one vertex of any type, with weight functions $\{ w_{ij}(\vec{n}) \}$ satisfying condition \eqref{A1}, the asymptotic proportion of type $i$ vertices satisfies
	\begin{align*}
	p_i(t) = \frac{Z_i(t)}{Z(t)} \to \frac{  u_i }{u_1+u_2+\dots+u_p } \text{ almost surely as } t \to \infty.
	\end{align*}
\end{theorem}
The next theorem asserts that the empirical degree distribution converges almost surely and identifies the limit.
\begin{theorem}
	\label{thm2}
	For the multi-type degree-based attachment model starting with one vertex of any type, with weight functions $\{ w_{ij}(\vec{n}) \}$ satisfying condition \eqref{A1}, we have that
	\begin{align*}
	\frac{ Z_i^{\vec{n} }(t)}{Z(t)} \to \alpha \frac{  u_i }{u_1+\dots+u_p } I_i(\vec{n}) \text{ almost surely as } t \to \infty
	\end{align*}
	where $I_i(\vec{n})$ satisfies the recursion 
	\begin{enumerate}[(i)]
		\item $I_i(\vec{0})  = \frac{1}{\alpha+w_{i1}(\vec{0}) +\dots+w_{ip}(\vec{0}) }$;\\
		\item if $|\vec{n}|>0$ then $ I_i(\vec{n})= \sum_{j=1}^{p} \frac{  w_{ij}(n_1,\ldots,n_j-1,\ldots,n_p)I_i(n_1,\ldots,n_j-1,\ldots,n_p) }{\alpha+w_{i1}(\vec{n}) +\dots+w_{ip}(\vec{n}) }$
		
	\end{enumerate}	
	with $w_{ij}(\vec{n}) = I_i(\vec{n})=0$ if $\min(n_1,\ldots, n_p)<0,$ $j=1,\ldots,p$.
\end{theorem}

\subsection{Results on Multi-type Linear Preferential Attachment}
A particularly interesting choice of weight functions is $w_{ij}(\vec{n}) = \gamma_{ij} (n_1+\dots+n_p)+\beta_{ij}$ where $\gamma_{ij}$ and $\beta_{ij}$ are positive constants. This is an extension of one-type linear preferential attachment.
We call this model multi-type linear preferential attachment based on total in-degree. As with its one-type counterpart the linear multi-type model exhibits a power law degree distribution as soon as $\gamma_{i1}+\dots+\gamma_{ip}>0 $.

\begin{theorem}
	\label{thm3}
	For the multi-type preferential attachment model starting with one vertex of any type, with weight functions $w_{ij}(\vec{n}) = \gamma_{ij} (n_1+\dots+n_p)+\beta_{ij}$ where $\gamma_{ij} \geq 0$ and $ \beta_{ij} > 0$, 
	$
	p_i(k) = \lim\limits_{t\to \infty} \frac{Z_i^{k}(t)}{Z(t)}
	$
	exists almost surely. Furthermore,
	\begin{align*}
	p_i(k) \sim  \begin{cases}
	C_1 \cdot k^{-(1+\frac{\alpha}{\gamma_{i1}+\dots+\gamma_{ip}})}\qquad &\text{if } \gamma_{i1}+\dots+\gamma_{ip}>0,\\
	C_2 \cdot e^{-k\log(1+\frac{\alpha}{\beta_{i1}+\dots+\beta_{ip}})} &\text{if } \gamma_{i1}+\dots+\gamma_{ip} = 0.
	\end{cases}
	\end{align*}
\end{theorem}
By summing over all type  $i$ vertices it is easy to see that the total asymptotic proportion of vertices with degree $k$ also follows a power law distribution, i.e.
\begin{align*}
p(k) = \lim\limits_{t\to \infty}	\frac{Z^{k}(t)}{Z(t)} = \sum_{i=1}^{p}\lim\limits_{t\to \infty} \frac{Z_i^{k}(t)}{Z(t)}  \sim  C \cdot k^{-(1+\frac{\alpha}{ \max_i \{ \gamma_{i1}+\dots+\gamma_{ip}  \} })}.
\end{align*}

The rest of the paper is organized as follows. In the next section we introduce the theory for multi-type branching processes needed to prove Theorem \ref{thm1}, \ref{thm2}, and \ref{thm3}. In Section 3 we prove Theorem \ref{thm1} and \ref{thm2}, and in Section 4 we prove Theorem \ref{thm3}. Finally, in Section 5, we investigate the results numerically.

\section{General Multi-type Branching Processes}
General branching processes has been extensively studied, and it is not our intention to summarize the results here, for this see e.g. the book by Jagers \cite{Jagers} for the single-type case and \cite{multiypeextension, NermanPHD} for multi-type generalizations. We shall, however, explain some of the concepts and results needed for proving Theorems \ref{thm1}, \ref{thm2}, and \ref{thm3}. When defining general multi-type branching processes we will follow the terminology of \cite{NermanPHD}, and in applying the theory we will mainly use the results of \cite{multiypeextension}. The main result needed from \cite{multiypeextension} is stated below as Theorem \ref{multitypethm}.

A $p$-type general branching process is a process where individuals can be of one of $p$ different types, and $i$-type individuals live for a random time $\lambda_i\in [0, \infty]$ during which they give birth to $j$-type individuals according to the points of a point process $\xi_{ij}$ defined on $\mathbb{R}_+$. Individuals live and reproduce independently of each other, but there is no restriction on the dependence between an individual's life time and reproduction process. Individuals of the same type have the same reproduction and life-time law.

We denote individuals by $x$ and their type by $\tau (x) \in \{1,2,\ldots, p \}$. If $x=(0,\tau_0; i_1,\tau_1; \ldots;i_n,\tau_n)$ then $x$ is the $i_n$:th child of type $\tau_n$ of $\ldots$ of the $i_1$:th child of type $\tau_1$ of the ancestor $0$, which is of type $\tau_0$. The space of possible individuals is denoted $\mathcal{J}$ and is defined by $ \mathcal{J} = \{ (0,\tau_0;i_1,\tau_1;\ldots;i_k,\tau_k);\ k\geq 0, i_j\in \{ 1,\ldots, p\},\ j \in \{1,\ldots, k \}   \}$
To each individual $x$ we assume there is a probability space $(\Omega_x, \mathbb{B}_x, \mathbb{P}_x)$ associated, on which $x$'s life-length $\lambda_x$, a characteristic $\phi_x$ (defined below and more stringent in \cite{multiypeextension}), and $x$'s reproduction $\xi_x = (\xi_x^1,\ldots,\xi_x^p)$ are defined. A characteristic is a product-measurable, separable (random) process $\phi: \Omega_x \times\mathbb{R} \to \mathbb{R}$ with $\phi(\omega, t) = 0$ if $t<0$: Let $\phi_x(t) = \phi(\omega_{\downarrow x}, t)$ where $\omega_{\downarrow x}$ is the outcome of the branching process starting with individual $x$ as ancestor. Hence, $\phi_x(t)$ is the score given to the individual $x$ of age $t$. Note that $\phi_x(t)$ is allowed to depend on $x$ and its whole progeny, a fact which is a major strength of random characteristics. However, we shall only use characteristics that depends only on the life history of $x$, not its entire progeny.

We can now define the $p$-type process with ancestor $x_0$ on the probability space 
\begin{align*}
(\Omega, \mathbb{B}, \mathbb{P}) = \prod_{x \in \mathcal{J}} (\Omega_x, \mathbb{B}_x, \mathbb{P}_x)
\end{align*}
through the birth times $\{\sigma_x\}$ defined by induction
\begin{align*}
&\sigma_{x_0} = 0, \text{ and if } x=(x';j_k,\tau_k)\\
&\sigma_x = \sigma_{x'}+\inf\{ t\geq 0,\  \xi_{x'}^{\tau_k}([0,t]) \geq j_k \}.
\end{align*}
We note that an individual $x$ who is never born will have $\sigma_x = \infty$. Now, let $Z^{\phi}(t)$ denote the total score of the population at time $t$, that is, 
\begin{align*}
Z^{\phi}(t) = \sum_{x\in \mathcal{J} } \phi_x (t-\sigma_x).
\end{align*}
We call $\{Z^{\phi}(t),\ t\geq 0 \}$ the general multi-type $\phi$-counted branching process.
When necessary we use the notation ${}_{i}Z^{\phi}(t)$ to emphasize that the process starts with one type $i$ individual. Different choices of $\phi$ give rise to different processes, e.g. if $\phi_x(t) = 1, t\geq 0$ then $Z^{\phi}(t)$ represents the number of individuals that have been born up to time $t$; and if $\phi_x(t) = 1\{ 0\leq t \leq \lambda_x \}$ then $Z^{\phi}(t)$ represents the number of individuals alive at time $t$. 
\begin{remark}
	Note that the process always starts with a single individual $x_0$ of type $\tau_0$ and that we omit this in the notation to simplify expressions.
\end{remark}

For all Borel measurable sets $A$ in $\mathbb{R}_+$ let $\mu_{ij}(A) = \mathbb{E}(\xi_x^j(A))$ if $x$ is of type $i$. It follows that $\mu_{ij}$ is a measure, see \cite[Lemma 1.1.1 ]{Point}. For each $\theta>0$, define the new measure $\mu_{ij}^{\theta}(A)=\mu_{ij}(A, \theta)$, $A\in \mathcal{B}(\mathbb{R}_+)$, through 
\begin{align*}
\mu_{ij}(A, \theta) = \int_{A} e^{-\theta s} \mu_{ij}(ds),\qquad \theta >0.
\end{align*}
Also, define
\begin{align*}
M(\theta) = [	\mu_{ij}([0,\infty], \theta)]_{i,j}.
\end{align*}
Following \cite{multiypeextension} we shall, for technical reasons, assume throughout this section that:

\begin{enumerate}[(C1)]
	\item For $i,j=1,2,\ldots, p$, the measure $\mu_{ij}$ is non-lattice, i.e. not concentrated on any set $\{b+\lambda \cdot \mathbb{Z}  ,\ \lambda \in \mathbb{R} \}$
	\item Either $M(0)$ has at least one infinite entry, or only finite entries and Perron root $\rho > 1$. Also $M(0)^n$ is assumed to have all positive entries (possibly infinite) for some $n\geq 1$.
	\item There exists an $\alpha > 0$ such that $M(\alpha)$ has only finite entries and Perron root $\rho = 1$, with corresponding left and right positive eigenvalues $u$ and $v$ normed such that $uv^T = \textbf{1}v^t=1.$
	\item For $i,j=1,2,\ldots,p$ we have that $\int_{0}^{\infty}ue^{-\alpha u}\mu_{ij}(du)<\infty$.
\end{enumerate}
The following assumption will only be in force when explicitly stated:
\begin{enumerate}[(C5)]
	\item There is some $\theta \in (0, \alpha)$ such that $M(\theta)$ has finite entries only. 
\end{enumerate}
\begin{remark}
	In the proof of Theorem \ref{thm1} we shall see that, for multi-type degree-based attachment, condition \eqref{A1} implies conditions (C2)-(C5) ((C1) follows from the model definition).
\end{remark}

We shall mainly be interested in results regarding the ratio of the process counted in different ways, i.e. $\lim\limits_{t\to \infty} \frac{Z^{\phi}(t)}{Z^{\psi}(t)}$.  However, one needs to put some restrictions on the random characteristics. Let $\phi$ be a random characteristic not $0$ a.e. and with paths in the Skorohod space $D(\mathbb{R})$ of right-continuous functions with finite left limits. Also assume that there exists a $\theta < \alpha$ such that, for $i=1,2,\ldots, p$,
\begin{equation}
\label{C6}
\mathbb{E}(\sup_{t\geq 0} e^{-\theta t}\phi_i(t) ) <\infty.  \tag{C6}
\end{equation}
\begin{remark}
	In the degree-based attachment setting we will work with bounded random characteristics and hence \eqref{C6} is trivially satisfied. 
\end{remark}
Finally, we can quote the result we need.
\begin{theorem}{\cite[Theorem 2.7]{multiypeextension}}
	\label{multitypethm}
	Assume that $\{Z(t)\}$ is a branching process with intensity measures $\{ \mu_{ij}  \}$ satisfying conditions (C1)-(C5). Furthermore assume that $\phi$ and $\psi$ are both random characteristics satisfying condition (C6). Then on the event that $\{Z(t)\to \infty \}$
	\begin{align*}
	\frac{Z^{\phi}(t)}{Z^{\psi}(t)} \to \frac{  \sum_{j=1}^{p} u_j \int_{0}^{\infty} \mathbb{E}(e^{-\alpha s} \phi_j(s))ds     }{ \sum_{j=1}^{p} u_j \int_{0}^{\infty} \mathbb{E}(e^{-\alpha s} \psi_j(s))ds    } \text{ a.s. as } t \to \infty.
	\end{align*}
\end{theorem}

\section{General Degree-Based Attachment Trees}
\label{General Framework}
In this section we will provide a general framework for deriving asymptotic ratio results on the multi-type degree-based attachment tree as defined in Section \ref{modeldef}. It is clear from the model definition that the vertex population evolves as a general multi-type branching process. Recall that, in addition to being a non-lattice process (which follows from model definition), the model is assumed to satisfy \eqref{A1} throughout.

Clearly of much importance is the matrix $\mu^*(\theta)$ as it is part of the condition \eqref{A1}. Hence, we need a way of calculating the integrals $\mu_{ij}^*(\theta)$ as defined in Section \ref{notation}.
One way of doing this is to calculate the Radon-Nikodym derivate of $\mu_{ij}$ with respect to the Lebesgue measure. An application of the fundamental theorem of calculus together with the Markov property of the model gives that $ \frac{d\mu_{ij}(t)}{dt} = \mathbb{E}(w_{ij}(\vec{\xi_i}(t) ))$, $ i,j=1,2,\ldots,p$ if the process $\xi_i(t)$ does not explode ((A1) ensures this is the case).
Hence, we can calculate $\mu^*_{ij}( \theta)$ through
\begin{align*}
\mu^*_{ij}(\theta) = \int_{0}^{\infty}e^{-\theta s} \mathbb{E}(w_{ij}(\vec{\xi}_i(s)))ds.
\end{align*}

\subsection{A Useful Integral}
In what follows we shall see that the integral
\begin{align*}
&I_i(\vec{n}, \theta) = \int_{0}^{\infty} e^{-\theta s} \mathbb{P}( \vec{\xi_i}(t)  = \vec{n} )ds,\qquad \vec{n}\in \mathbb{N}^p\\
\end{align*}
is very useful in deriving Theorem \ref{thm1}, \ref{thm2}, and \ref{thm3}. This is because many quantities of interest can be written as a sum of $I_i(\vec{n}, \theta)$ over some set. For instance,
\begin{align*}
\mu_{ij}^*(\theta) = \sum_{\vec{n} \in \mathbb{N}^p}w_{ij}(\vec{n})I_i(\vec{n},\theta).
\end{align*}
We shall therefore spend some time investigating $I_i(\vec{n}) \equiv I_i(\vec{n}, \alpha)$, and deriving a recursion for it. 
\begin{lemma}
	\label{I-prop}
	Let $\vec{\xi_i}(t) = (\xi_{i1}(t),\dots,\xi_{ip}(t))$ where $\xi_{ij}(t)$ are the counting processes of a multi-type degree-based attachment model, where $\xi_{ij}(t)$ does not explode in finite time. For $\vec{n} =(n_1,\ldots, n_p) \in \mathbb{N}^p$ define 
	\begin{align*}
	I_i(\vec{n}) = \int_{0}^{\infty} e^{-\alpha s} \mathbb{P}( \vec{\xi_i}(t) =\vec{n} )ds.
	\end{align*}
	Then
	\begin{enumerate}[(i)]
		\item $I_i(\vec{0})  = \frac{1}{\alpha+w_{i1}(\vec{0}) +\dots+w_{ip}(\vec{0}) }$;\\
		\item if $|\vec{n}|>0$ then $ I_i(\vec{n})= \sum_{j=1}^{p} \frac{  w_{ij}(n_1,\ldots,n_j-1,\ldots,n_p)I_i(n_1,\ldots,n_j-1,\ldots,n_p) }{\alpha+w_{i1}(\vec{n}) +\dots+w_{ip}(\vec{n}) }.$
	\end{enumerate}	
	where $w_{i1}(\vec{n}) = w_{ip}(\vec{n}) = I_i(\vec{n})=0$ if $\min(n_1,\ldots, n_p)<0$, $i,j=1,2,\ldots,p.$
	
\end{lemma}
\begin{proof}
	By definition, $\vec{\xi_i}(t)$ is a Markov process on $\mathbb{N}^p$ and, by assumption, it does not explode in finite time. Hence, it satisfies the Kolmogorov forward equations. 
	
	Assuming that $w_{ij}(\vec{n}) = 0$ if $\min(n_1,\ldots, n_p)<0$, $i,j=1,2,\ldots,p.$, we get
	\begin{multline*}
	\frac{d}{dt} \mathbb{P}( \vec{\xi_i}(t) = \vec{n} ) = \sum_{j = 1}^{p}  w_{ij}(n_1,\ldots, n_j-1,\ldots, n_p) \mathbb{P}(\vec{\xi_i}(t) = (n_1,\ldots, n_j-1,\ldots, n_p))   \\
	- (w_{i1}( \vec{n})+\dots+w_{ip}(\vec{n}))\mathbb{P}(\vec{\xi_i}(t) = \vec{n} ).
	\end{multline*}
	
	Using this together with integration by parts yields
	\begin{multline*}
	I_i(\vec{0}) = \int_{0}^{\infty}e^{-\alpha s}\mathbb{P}( \vec{\xi_i}(s)=\vec{n}  )ds =\frac{1}{\alpha}+ \frac{1}{\alpha}\int_{0}^{\infty}e^{-\alpha s}	\frac{d}{ds} \mathbb{P}( \vec{\xi_i}(s) = \vec{n} )ds\\
	=\frac{1}{\alpha} -\frac{1}{\alpha} (w_{i1}(\vec{n})+\dots+w_{ip}(\vec{n}))I_i(\vec{n})
	\end{multline*}
	and for arbitrary $\vec{n}$ we get
	\begin{multline*}
	I_i(\vec{n}) = \int_{0}^{\infty}e^{-\alpha s}\mathbb{P}( \vec{\xi_i}(s)=\vec{n}  )ds = \frac{1}{\alpha}\int_{0}^{\infty}e^{-\alpha s}	\frac{d}{ds} \mathbb{P}( \vec{\xi_i}(s) = \vec{n} )ds\\
	=\frac{1}{\alpha} \sum_{j = 1}^{p}w_{ij}(n_1,\ldots, n_j-1,\ldots, n_p)I_i(n_1,\ldots, n_j-1,\ldots, n_p)  \\
	-\frac{1}{\alpha} (w_{i1}(\vec{n})+\dots+w_{ip}(\vec{n}))I_i(\vec{n}) .
	\end{multline*}
	Finally, solving for $I_i(\vec{n})$ gives
	\begin{align*}
	I_i(\vec{n})= \sum_{j=1}^{p} \frac{  w_{ij}(n_1,\ldots,n_j-1,\ldots,n_p)I_i(n_1,\ldots,n_j-1,\ldots,n_p) }{\alpha+w_{i1}(\vec{n}) +\dots+w_{ip}(\vec{n}) }.
	\end{align*}
\end{proof}

We can now prove Theorem \ref{thm1} and \ref{thm2}.
\begin{proof}[Proof of Theorem \ref{thm1}]
	Let $\phi_x(t) = 1\{ \tau(x)=i \} $ be the random characteristic assigning type $i$ vertices score $1$ and type $j \neq i$ vertices score $0$. The branching process $Z_i(t) = Z^{\phi}(t)$ starting with 1 vertex of \textit{any type} represents the number of type $i$ vertices at time $t$.
	
	Let $\psi_x(t) \equiv 1$ and put $Z(t) = Z^{\psi}(t)$---this is the original branching process counting the number of vertices alive at time $t$. 
	
	We want to apply Theorem \ref{multitypethm} to the ratio $\frac{Z_i(t)}{Z(t)}$ and need to check that conditions (C1)-(C6) are satisfied.
	The random characteristics trivially satisfy the condition (C6) of Theorem \ref{multitypethm} and, since $\mu_{ij}(t)$ are non-lattice measures by design, also condition (C1) is fulfilled. 
	
	It remains to prove that condition \eqref{A1} implies conditions (C2)-(C5). By \eqref{A1}, the Perron root exists for all $\lambda \geq \theta_0$. Since $\rho(\mu^*(\theta_0)   )>1$ for some $\theta_0 >0$ and as $\rho(\mu^*(\theta)   )$ is a decreasing function, $\rho(\mu^*(0)  )$ is larger than $1$ (or an entry is infinite) and condition (C2) is satisfied. Condition (C3) follows from (A1) since the Perron root is continuous and decreasing to 0 \cite[Lemma 9.1]{Mode}. Condition (C5) follows trivially from (A1).
	
	Left to show is that (C4) is satisfied, i.e. that
	\begin{align*}
	\int_{0}^{\infty} ue^{-\alpha u} \mu_{ij}(du) < \infty.
	\end{align*}
	It follows from \eqref{A1} that $\alpha > \theta_0$ and $\int_{0}^{\infty} e^{-\theta_0 u} \mu_{ij}(du) < \infty$. For large enough $u$, we have that $u e^{-\alpha u} < e^{-\theta_0 u}$ and therefore that $\int_{0}^{\infty} ue^{-\alpha u} \mu_{ij}(du)<\infty$.
	
	All conditions of Theorem \ref{multitypethm} are satisfied and applying it to $\lim\limits_{t \to \infty} \frac{Z_i(t)}{Z(t)}$ yields
	\begin{align*}
	\lim\limits_{t \to \infty} \frac{Z_i(t)}{Z(t)} = \frac{u_i}{u_1+\dots+u_p} 
	\end{align*}
\end{proof} 
The proof of  Theorem \ref{thm2} is similar.
\begin{proof}[Proof of Theorem \ref{thm2}]
	Again we wish to apply Theorem \ref{multitypethm}. Let $\phi_x(t) = 1\{ \tau(x)=i,\ \xi(t)=\vec{n} \} $ be the random characteristic assigning score 1 to type $i$ vertices with $n_k$ children of type $k$, $k=1,\ldots, p$. Let $Z_i^{\vec{n}}(t) = Z^{\phi}(t)$ be the branching process associated with this characteristic. Similarly let $\psi_x(t) =1\{t \geq 0 \}$ and $Z(t) = Z^{\psi}(t)$ be the branching process counting the number of vertices born/alive at time $t$. 
	
	The proof of Theorem \ref{thm1} shows that we can apply Theorem \ref{multitypethm} and we get
	\begin{align*}
	\lim\limits_{t \to \infty} \frac{Z_i^{\vec{n}}(t)}{Z(t)} = \frac{u_i\int_{0}^{\infty} e^{-\alpha s  } \mathbb{P}( \xi_i(s)=\vec{n})ds}{(u_1+\ldots+u_p)/\alpha}= \alpha\frac{ u_i }{u_1+\ldots+u_p} I_i(\vec{n}).
	\end{align*}
	The latter part of the theorem follows directly from Lemma \ref{I-prop}.
\end{proof} 

\section{Multi-type Linear Preferential Attachment Trees}
In this section the theory from Section \ref{General Framework} is applied to investigate the limiting behavior of the degree distribution, and asymptotic composition of the vertex population, for a specific family of weight functions $w_{ij}(\vec{n})$. We consider the case when the weight functions are given by $w_{ij}(\vec{n}) = \gamma_{ij} (n_1+\dots+n_p)+\beta_{ij}$, with $\gamma_{ij}$ and $\beta_{ij}$ being positive constants.
We call this model multi-type linear preferential attachment based on total in-degree. The main purpose of this section is to prove Theorem \ref{thm3}. 

First we will need to investigate if condition \eqref{A1} is satisfied. Hence, we will need to calculate $\mu_{ij}^{*}(\alpha)$. In order to do so we need the fact that $\frac{d\mu_{ij}(t)}{dt} =  \mathbb{E}( w_{ij}( \vec{\xi_i}(t)  )  )$ and Lemma 1, both of which assumes that $\vec{\xi_i}(t)$ does not explode. Since $\xi_i$ is Markov non-explosion follows the condition
\begin{align*}
	\sum_{\vec{n} \in \mathbb{N}^p } \frac{1}{w( \vec{n}  )} =	\sum_{\vec{n} \in \mathbb{N}^p }  \frac{1}{ \sum_{i,j} \gamma_{ij} (n_1+\dots+n_p)+\sum_{i,j}\beta_{ij}} = \infty.
\end{align*}

Now that we know $\frac{d\mu_{ij}(t)}{dt} =  \mathbb{E}( w_{ij}( \vec{\xi_i}(t)  )  )$, note that the density of $\mu_{ij}$ is given by
\begin{align*}
\frac{d\mu_{ij}(t)}{dt} =  \mathbb{E}( w_{ij}( \vec{\xi_i}(t)  )  ) = \mathbb{E}( w_{ij}( \xi_{i1}(t) + \dots+\xi_{ip}(t)  )  ) = \sum_{k=0}^{\infty} w_{ij}(k)\mathbb{P}( \xi_i^{\Sigma}(t)=k  )
\end{align*}
where $\xi^{\Sigma}_i(t) = \xi_{i1}(t) + \dots+\xi_{ip}(t) $.
In deriving an expression for $\mu_{ij}^{*}(\alpha)$ it will first be useful to study $I_i(k)=\sum_{n_1+\dots+n_p=k}I_i(\vec{n}) = \int_{0}^{\infty} e^{-\alpha s }\mathbb{P}(\xi^{\Sigma}_i(s)=k) ds$. 
\begin{lemma}
	\label{I-k}
	If the weight functions of a non-explosive multi-type preferential attachment tree satisfy $w_{ij}(\vec{n}) = w_{ij}(n_1+\dots + n_p)$ then
	\begin{align*}
	I_i(k)= \frac{1}{\alpha+w_{i1}(k)+\dots+w_{ip}(k)}\prod_{n=0}^{k-1}\frac{w_{i1}(n)+\dots+w_{ip}(n)}{\alpha+w_{i1}(n)+\dots+w_{ip}(n)}
	\end{align*}
	where an empty product is defined to equal 1.
\end{lemma}
\begin{proof}
	We prove the formula by induction. First note that it holds $k=0$. Assume that it holds for $k>0$. By Lemma \ref{I-prop} and the induction assumption
	\begin{multline*}
	I_i(k+1)=\sum_{n_1+\dots+n_p=k+1}I_i(\vec{n}) 
	=\sum_{n_1+\dots+n_p=k+1} \sum_{j=1}^{p} \frac{  w_{ij}(k)I_i(n_1,\ldots,n_j-1,\ldots,n_p) }{\alpha+w_{i1}(k+1) +\dots+w_{ip}(k+1) }\\
	= \sum_{j=1}^{p } \frac{w_{ij}(k)}{\alpha+w_{i1}(k+1) +\dots+w_{ip}(k+1)}\sum_{n_1+\dots+n_p=k+1}  I_i(n_1,\ldots,n_j-1,\ldots,n_p). 
	\end{multline*}
	Since, $I_i(n_1,\ldots,n_j-1,\ldots,n_p) = 0$ if $n_j=0$ we get that $\sum_{n_1+\dots+n_p=k+1}  I_i(n_1,\ldots,n_j-1,\ldots,n_p) = \sum_{n_1+\dots+n_p=k}  I_i(\vec{n})$. Carrying on the calculations yields
	\begin{multline*}
	I_i(k+1)= \sum_{j=1}^{p } \frac{w_{ij}(k)}{\alpha+w_{i1}(k+1) +\dots+w_{ip}(1+k)} I_i(k) 
	= \frac{(\sum_{j=1}^{p} w_{ij}(k))I_i(k)}{\alpha+w_{i1}(k+1) +\dots+w_{ip}(k+1)}
	\end{multline*}
	which proves the formula.
\end{proof}
This result is easily extended to an expression for $\mu^*_{ij}(\theta)$ given that the weight functions depend only on the total in-degree.
\begin{corollary}
	\label{laplace-corr}
	If the weight functions of a non-explosive multi-type preferential attachment tree satisfy $w_{ij}(\vec{n}) = w_{ij}(n_1+\dots+n_p)$, then for $\theta > 0$
	\begin{align*}
	\mu^*_{ij}(\theta) = \sum_{k=0}^{\infty} \frac{w_{i1}(k)}{\theta+w_{i1}(k)+\dots+w_{ip}(k)}\prod_{n=0}^{k-1}\frac{w_{i1}(n)+\dots+w_{ip}(n)}{\theta+w_{i1}(n)+\dots+w_{ip}(n)}
	\end{align*}
\end{corollary}
\begin{proof}
	Let $\xi^{\Sigma}_i(t)=\xi_{i1}(t)+\dots+\xi_{ip}(t)$ then
	\begin{multline*}
	\mu_{ij}^*(\theta) = \int_{0}^{\infty} e^{-\theta s } \mu_{ij}(ds) 
	= \int_{0}^{\infty} e^{-\theta s } \sum_{k=0}^{\infty}w_{ij}(k)\mathbb{P}(\xi^{\Sigma}_i(s)=k) ds\\
	=  \sum_{k=0}^{\infty}w_{ij}(k)\int_{0}^{\infty} e^{-\theta s }\mathbb{P}(\xi^{\Sigma}_i(s)=k) ds =\sum_{k=0}^{\infty} w_{ij}(k)I_i(k, \theta)
	\end{multline*}
	where the second to last equality follows from monotone convergence. An application of Lemma \ref{I-k} yields the desired result.
\end{proof}
\begin{remark}
	The above expression for $\mu^*_{ij}(\theta)$ is in accordance with the results of \cite{RSA} and \cite{Deijfen} when $p=1$. Also, note that $\mu_{ij}^*(\theta)$ may well be infinite if $\theta<\theta_0$, with $\theta_0$ as in \eqref{A1}.
\end{remark}

Given that the weight functions are linear in total degree we can derive an even more explicit result for the Laplace transform $\mu_{ij}^*(\theta)$.
\begin{corollary}
	\label{laplacetransform}
	If the weight functions of the multi-type preferential attachment model satisfy $w_{ij}(\vec{n}) = \gamma_{ij}(n_1+\dots+n_p) + \beta_{ij} $, with $\gamma_{ij}\geq 0,\beta_{ij}> 0$ and $\gamma_{i1}+\dots+\gamma_{ip}>0$, then
	\begin{align*}
	\mu^*_{ij}(\theta) = \int_{0}^{\infty} e^{-\theta s }\mu_{ij}(ds) = \begin{cases}
	\infty &\text{ if } 0 < \theta \leq\gamma_{i1}+\dots+\gamma_{ip},\\
	\frac{\beta_{ij}}{\theta} + \frac{\gamma_{ij} (\beta_{i1} +\dots+ \beta_{ip})   }{  \theta(\theta-(\gamma_{i1}+\dots+\gamma_{ip}))   } &\text{ if }\theta > \gamma_{i1}+\dots+\gamma_{ip}\\
	\end{cases}
	\end{align*}
\end{corollary}
\begin{proof}
	For convenience set $\gamma = \gamma_{i1}+\dots+\gamma_{ip}$ and $\beta = \beta_{i1}+\dots+\beta_{ip}$. First assume that $\gamma_{ij}>0$.
	By Corollary \ref{laplace-corr}, whenever the Laplace transform exists, we have
	\begin{align*}
	\mu_{ij}^*(\theta) = \sum_{k=0}^{\infty} \frac{w_{i1}(k)}{\theta+w_{i1}(k)+\dots+w_{ip}(k)}\prod_{n=0}^{k-1}\frac{w_{i1}(n)+\dots+w_{ip}(n)}{\theta+w_{i1}(n)+\dots+w_{ip}(n)} = \sum_{k=0}^{\infty}\frac{ \gamma_{ij}k+\beta_{ij} }{\gamma k + \theta + \beta} \prod_{n=0}^{k-1} \frac{\gamma n + \beta}{\gamma n +\theta + \beta}.
	\end{align*}
	Using that $\prod_{i=0}^{k}(i+c) = \frac{\Gamma (k+1+c)}{\Gamma(c)}$ we get
	\begin{align*}
	\mu_{ij}^*(\theta)
	= \frac{\gamma_{ij} \Gamma(\theta/\gamma+\beta/\gamma)  }{\gamma \Gamma(\beta /\gamma )}\left(  \sum_{k=0}^{\infty} \frac{   \Gamma(k+1+\beta/\gamma)    }{\Gamma(k+1+\theta/\gamma+\beta/\gamma)  }+(\beta_{ij}/\gamma_{ij}-\beta/\gamma)\sum_{k=0}^{\infty} \frac{ \Gamma(k+\beta/\gamma)}{ \Gamma( k+1+\theta/\gamma + \beta/\gamma  )}\right).
	\end{align*}
	Using the relation $\sum_{k=0}^{n}  \frac{\Gamma(k+a)}{\Gamma(k+c)} = \frac{1}{1+a-c}\left(\frac{\Gamma(n+1+a)}{\Gamma(n+c)}-\frac{\Gamma(a)}{\Gamma(c-1)} \right)  $, valid for all real numbers $a,c$ (see \cite{Deijfen}), together with Stirling's formula $\Gamma(k+c)/\Gamma(k) \sim k^c$ we get
	\begin{align*}
	 \sum_{k=0}^{\infty} \frac{\Gamma(k+a)}{\Gamma(k+c)} =\frac{1}{c-a-1} \frac{\Gamma(a)}{\Gamma(c-1)}, \text{ if } 1+a<c.
	\end{align*}
	We note that the sum diverges if $1+a\geq c$. Hence, with $a=1+\beta/\gamma$ and $c=1+\theta/\gamma+\beta/\gamma$, we get that $\mu_{ij}^*(\theta)$ converges if and only if $\theta > \gamma = \gamma_{i1}+\dots+\gamma_{ip}$. For $\theta > \gamma$, we get
	\begin{align*}
	\mu_{ij}^*(\theta) = 	\frac{\gamma_{ij} \Gamma(\theta/\gamma+\beta/\gamma)  }{\gamma \Gamma(\beta /\gamma )}\left( \frac{\Gamma(1+\beta / \gamma )}{(\theta / \gamma -1 ) \Gamma(\theta / \gamma + \beta/ \gamma)}  +(\beta_{ij}/\gamma_{ij}-\beta/\gamma) \frac{ \Gamma(\beta / \gamma)   }{  \theta / \gamma \Gamma(\theta / \gamma + \beta / \gamma) }     \right).
	\end{align*}
	All together we then have (again using $\prod_{i=0}^{k}(i+c) = \frac{\Gamma (k+1+c)}{\Gamma(c)}$),
	\begin{align*}
	\mu_{ij}^*(\theta) = \begin{cases}
	\infty, &\text{ if } \theta \leq \gamma_{i1}+\dots+\gamma_{ip} \\
	\frac{\beta_{ij}}{\theta} + \frac{\gamma_{ij}(\beta_{i1}+\dots+\beta_{i2})}{\theta (\theta-(\gamma_{i1}+\dots+\gamma_{ip}))}, &\text{ if } \theta > \gamma_{i1}+\dots+\gamma_{ip}.
	\end{cases}
	\end{align*}
	Next assume that $\gamma_{ij} = 0$. Then $\mu_{ij}^*(\theta)$ is just the Laplace transform of the intensity measure of a Poisson process which is in accordance with the formula, i.e. $\mu_{ij}^*(\theta) = \frac{\beta_{ij}}{\theta}$.
\end{proof}
Corollary \ref{laplacetransform} immediately implies that condition \eqref{A1} is satisfied for multi-type linear preferential attachment since
\begin{align*}
\min_i \sum_{j=1}^{p} \mu^*_{ij}(\theta) \leq \rho(\mu^*(\theta))
\end{align*}
and $\lim\limits_{\theta \downarrow \gamma_{i1}+\dots+\gamma_{ip} } \mu^*_{ij}(\theta) = \infty$.

Using Corollary \ref{laplacetransform}, it is  possible to calculate the Perron root of $\mu^*(\theta)$ and Malthusian parameter $\alpha$ as well as the corresponding eigenvectors. However, already for $p=2$ these expressions are rather complicated and are better calculated numerically.

We are finally ready to prove Theorem \ref{thm3}. 
\begin{proof}[Proof of Theorem \ref{thm3}]
	Given that the model starts with one vertex we have by Theorem \ref{multitypethm} that
	\begin{align*}
	\frac{Z_i^k(t)}{Z_i(t)} \to \alpha \frac{ u_i}{u_1+\dots+u_p} I_i(k) \text{ as } t \to \infty
	\end{align*}
	i.e. the proportion of type $i$ vertices with $k$ children in total converges to the expression on the right-hand side above. Again let $\gamma = \gamma_{i1}+\dots+\gamma_{ip}$ and $\beta = \beta_{i1}+\dots + \beta_{ip}$. First assume that $\gamma>0$. Then by Lemma \ref{I-k}
	\begin{align*}
	I_i(k) = \frac{\prod_{n=0}^{k-1} \gamma n+\beta   }{\prod_{n=0}^{k} \alpha + \gamma n+\beta } = \frac{\Gamma(\frac{\alpha+\beta}{\gamma})}{\gamma\Gamma(\frac{\beta}{\gamma})} \frac{ \Gamma(k+\frac{\beta }{\gamma})   }{ \Gamma(k+1+\frac{\alpha+\beta }{\gamma})   }.
	\end{align*}
	By the same methods as in the proof of Corollary \ref{laplacetransform}, we get
	\begin{align*}
	I_i(k) \sim C \cdot k^{-(1+\frac{\alpha}{\gamma})}
	\end{align*}
	and the first part of the theorem is proved.
	
	Secondly, assume that  $\gamma_{i1}+\dots+\gamma_{ip}=0$. Then
	\begin{align*}
	I_i(k) = \frac{1}{\alpha +\beta } \left( \frac{\beta}{\alpha +\beta } \right)^k = \frac{1}{\alpha +\beta } e^{-k(\log(1+\frac{\alpha}{   \beta }))}
	\end{align*}
	and the second part of the theorem is proved.
\end{proof}

\newpage
\section{Numerical Examples}
\label{Numerical}
We now numerically investigate the behavior of the asymptotic composition of the vertex population as well as the exponent of the empirical degree distribution for some natural examples.

Consider first the two-type linear preferential attachment model with $w_{11}(k) = \gamma_{11}k+1 $ and $w_{12}(k) =w_{21}(k) = w_{22}(k) = k+1$. We now vary the rate at which type 1 vertices generate new type 1 vertices, i.e. $\gamma_{11}$, while keeping everything else fixed. For $\gamma_{11}<1$ we expect fewer type 1 than type 2 vertices in the graph and the opposite for $\gamma_{11}>1$. This is indeed true, see Figure 1. Recall from Theorem \ref{thm3} that the asymptotic behavior of the empirical degree distribution is given by
\begin{align}
\label{numeq}
\lim\limits_{t\to \infty}	\frac{Z_i^{k}(t)}{Z(t)} \sim 
C_1 \cdot k^{-(1+\frac{\alpha}{\gamma_{i1}+\gamma_{i2}})}\qquad &\text{if } \gamma_{i1}+\gamma_{i2}>0.
\end{align}
Hence, a lower absolute value of the power law exponent corresponds to a heavier tail of the degree distribution. Clearly, for large values of $\gamma_{11}$, type 1 will have a heavier tail, and this can be observed in Figure 1.
\begin{figure}[H]
	\includegraphics[width = 11cm, height = 6cm]{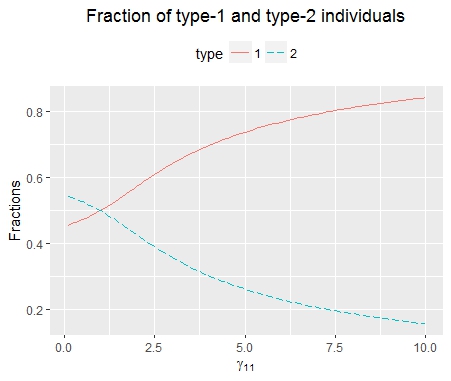}
	\includegraphics[width = 11cm, height = 6cm]{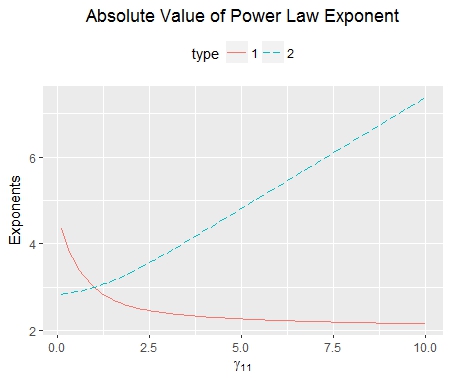}
	\caption{$w_{11}(k) = \gamma_{11}k+1 $ and $w_{12}(k) =w_{21}(k) = w_{22}(k) = k+1$}
\end{figure}

Next consider the case when $w_{12}(k) = \gamma_{12}k+1 $ and $w_{11}(k) =w_{21}(k) = w_{22}(k) = k+1$. We now vary the rate at which type 1 vertices generate new type 2 vertices while keeping everything else fixed. In Figure 2 we can see that for $\gamma_{12}<1$ there is a majority of type 1 vertices, while for $\gamma_{12}>1$ there is a majority of type 2 vertices. In fact, the qualitative behavior of the asymptotic composition is the opposite of previous model, compare Figure 1 and 2. Although there are more type 2 vertices for values of $\gamma_{12}>1$ we note that it is type 1 vertices that generate them. Hence, there should still be more type 1 vertices with high total degree. This is indeed true, and can be observed in Figure 2. We note that power law exponents are the same as for the previous model---this follows from \eqref{numeq}.

\begin{figure}[H]
	\includegraphics[width = 11cm, height = 6cm]{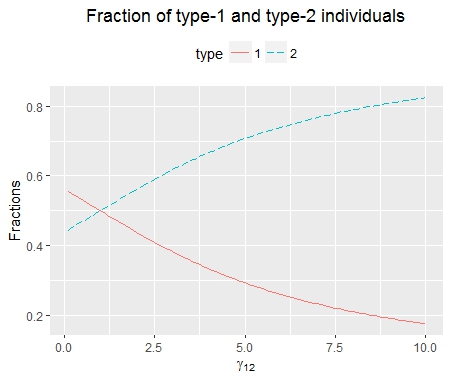}
	\includegraphics[width = 11cm, height = 6cm]{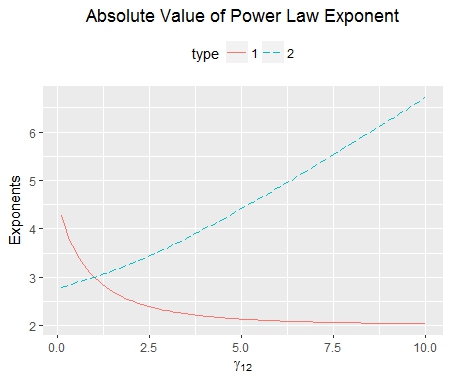}
	\caption{$w_{12}(k) = \gamma_{12}k+1 $ and $w_{11}(k) =w_{21}(k) = w_{22}(k) = k+1$}
\end{figure}

The value of the parameters $\{\beta_{ij} \}$ influences the power law exponents through the Malthusian parameter, and the asymptotic composition through the left eigenvector $u$. For the last example we consider the model where $w_{11}(k) = \gamma_{11} k + 1$, $w_{12}(k)=k+10$, $=w_{21}(k) = k+1$ and $w_{22}(k) = k + 10$. Hence, the model has larger constants for generating type 2 vertices. Even for large values of $\gamma_{11}$ there are still more type 2 than type 1 vertices in the graph, i.e. the constants $\beta_{12} = 10$ and $\beta_{22} = 10$ have a large influence on the asymptotic composition of the vertex population. Comparing Figure 3 with Figure 1 and 2 we see that degree distributions have thinner tails. This is because the larger values of the constants  $\beta_{12}$ and  $\beta_{21}$ weaken the preferential attachment mechanic in that it puts more weight on vertices with lower degree, e.g. degree 0 vertices have rate 10 instead of 1 as in the previous two examples.

Comparing the figures above we conclude that $\{\beta_{ij} \}$ has a large influence on asymptotic composition of the vertex population, and less influence on the degree distributions. 
\begin{figure}[H]
	\includegraphics[width = 11cm, height = 6cm]{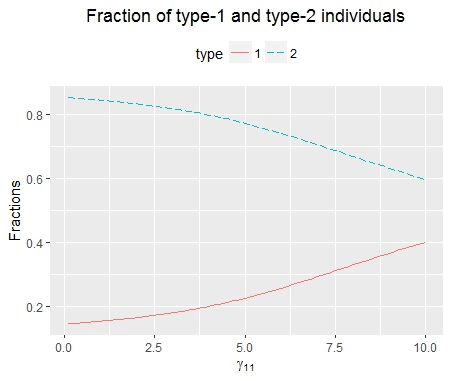}
	\includegraphics[width = 11cm, height = 6cm]{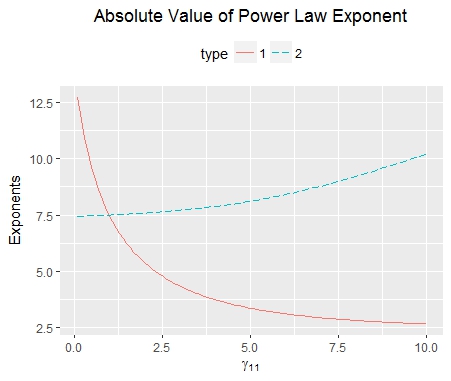}
	\caption{$w_{11}(k) = \gamma_{11} k + 1$, $w_{12}(k)=k+10$, $=w_{21}(k) = k+1$ and $w_{22}(k) = k + 10$.}
\end{figure}

\section{Further Work}
There are special cases of the model which can be studied further. For instance, the framework can be applied to the case when the rate functions are given by $w_{ij}(\vec{n})=w_{ij}(n_j)$, i.e. when the reproduction processes of a vertex are independent. Using the framework one can identify the limit of $\lim\limits_{t\to \infty}\frac{Z_i^{\phi^j}(t)}{Z(t)}$, where $Z_i^{\phi^j}(t)$ is the number of vertices at time $t$ of type $i$ with $k$ type $j$ children. For instance, if $w_{ij}(\vec{n}) = \gamma_{ij}n_j+\beta_{ij}$ then the asymptotic behavior in $k$ of this fraction is given by
\begin{align*}
\lim\limits_{t\to \infty}	\frac{Z_i^{\phi^j}(t)}{Z(t)} \sim C \cdot k^{-(1+\frac{\alpha}{\gamma_{ij}})}, \qquad C\in \mathbb{R}.
\end{align*}
This follow by noting that, by Theorem \ref{multitypethm}, we have
\begin{align*}
\lim\limits_{t\to \infty}	\frac{Z_i^{\phi^j}(t)}{Z(t)} = C_1\alpha \int_{0}^{\infty}e^{-\alpha s} \mathbb{P}(\xi_{ij}(s) = k )ds = C \sum_{\vec{n}:\ n_j=k} I_i(\vec{n})
\end{align*}
and, by Lemma \ref{I-prop} and the proof of Corollary \ref{laplacetransform}, we get
\begin{multline*}
\sum_{\vec{n}:\ n_j = k} I_i(\vec{n}) 
=  \frac{1}{\alpha+w_{ij}(k)} \prod_{n=0}^{k-1} \frac{w_{ij}(k)}{\alpha+w_{ij}(k)} \\
=  \frac{1}{\alpha+\gamma_{ij}k+\beta_{ij}} \prod_{n=0}^{k-1} \frac{\gamma_{ij}n+\beta_{ij}}{\alpha+\gamma_{ij}n+\beta_{ij}} = \frac{ \Gamma(\frac{\alpha+\beta_{ij}}{\gamma_{ij}})  }{\gamma_{ij} \Gamma(\frac{\beta_{ij}}{\gamma_{ij}}) } \frac{\Gamma(k+\frac{\beta_{ij}}{\gamma_{ij}})}{\Gamma(k+1+\frac{\alpha+\beta_{ij}}{\gamma_{ij}})}\sim C \cdot k^{-(1+\frac{\alpha}{\gamma_{ij}})}.
\end{multline*}

However, an expression for how the fraction of type $i$ vertices with $k$ children in total behaves as $k$ grows large does not follow easily from the framework, and is left as an open problem.

There are also extensions of the model which can be studied. Following \cite{Deijfen} we could allow for vertex death. The framework developed here can not be directly applied to this situation, but it should be possible to extend to allow for vertex death. With no vertex death the preferential attachment graph is a tree, and questions about the largest component are not interesting. However, with vertex death the graph becomes a forest, and questions about the largest component arise. Will a large component emerge? If so, how large is it?

\section{Acknowledgments}
I would like to extend my gratitude towards my supervisor Mia Deijfen for introducing me to the model, and for helpful comments on the manuscript. I would also like to thank KaYin Leung and the Journal Club at Stockholm University for helpful feedback, as well as Professor Olle Nerman for pointing out the newest results on general multi-type branching processes. 
\bibliography{articleRef}
\bibliographystyle{plain}

\end{document}